\documentclass[12pt]{amsart}
\usepackage{amssymb,amsmath,amsthm}

\usepackage{fullpage}
\usepackage{colonequals}
\usepackage{hyperref}
\hypersetup{colorlinks=true,urlcolor=blue,citecolor=blue,linkcolor=blue}
\usepackage{mathtools}
\usepackage{comment}
\usepackage[percent]{overpic}

\numberwithin{equation}{section}
\newtheorem{theorem}[equation]{Theorem}
\newtheorem{thm}[equation]{Theorem}
\newtheorem{lemma}[equation]{Lemma}
\newtheorem{lem}[equation]{Lemma}

\newtheorem{prop}[equation]{Proposition}
\newtheorem{cor}[equation]{Corollary}

\theoremstyle{definition}

\theoremstyle{remark}

\newcommand{\jv}[1]{{\color{red} \textsf{[JV: #1]}}}

\DeclareMathOperator{\hht}{ht}
\DeclareMathOperator{\GL}{GL}

\newcommand{\dD}{\mathrm{d}}

\DeclarePairedDelimiter\abs{\lvert\hspace{0.1ex}}{\rvert}

\numberwithin{equation}{section}

\newcommand{\FF}{\mathbb{F}}
\newcommand{\ZZ}{\mathbb{Z}}
\newcommand{\QQ}{\mathbb{Q}}
\newcommand{\RR}{\mathbb{R}}

\newcommand{\F}{\FF}
\newcommand{\Q}{\QQ}
\newcommand{\R}{\RR}
\newcommand{\Z}{\ZZ}

\newcommand{\scrE}{\mathcal E}

\DeclareMathOperator{\Gal}{Gal}

\newcommand{\defi}[1]{\textsf{#1}} 

\begin{document}

\title[Counting 3-isogenies]{Counting elliptic curves with \\ an isogeny of degree three}

\author{Maggie Pizzo}
\address{Mathematics Department, Dartmouth College, Hanover, NH 03755}
\email{Magdalene.R.Pizzo.19@dartmouth.edu}

\author{Carl Pomerance}
\address{Mathematics Department, Dartmouth College, Hanover, NH 03755}
\email{carl.pomerance@dartmouth.edu}

\author{John Voight}
\address{Mathematics Department, Dartmouth College, Hanover, NH 03755}
\email{jvoight@gmail.com}


\maketitle

\begin{abstract}
We count by height the number of elliptic curves over $\Q$ that possess an isogeny of degree $3$.
\end{abstract}

\section{Introduction} \label{intro}

Torsion subgroups of elliptic curves have long been an object of fascination for mathematicians.  By work of Duke \cite{Duke}, elliptic curves over $\Q$ with nontrivial torsion are comparatively rare.  Recently, Harron--Snowden \cite{HS} have refined this result by counting elliptic curves over $\Q$ with prescribed torsion, as follows.  Every elliptic curve $E$ over $\Q$ is defined uniquely up to isomorphism by an equation of the form 
\begin{equation} \label{eqn:Eyw2}
E \colon y^2=f(x)=x^3+Ax+B
\end{equation} 
with $A,B \in \Z$ such that $4A^3+27B^2 \neq 0$ and there is no prime $\ell$ such that $\ell^4 \mid A$ and $\ell^6 \mid B$.  We define the \defi{height} of such $E$ by 
\begin{equation}
\hht(E) \colonequals \max(\abs{4A^3},\abs{27B^2}).
\end{equation}
For $G$ a possible torsion subgroup (allowed by Mazur's theorem \cite{Mazur}), Harron--Snowden \cite[Theorem 1.5]{HS} prove that
\[ \#\{E : \hht(E) \leq X \text{ and } E(\Q)_{\textup{tors}} \simeq G\} \asymp X^{1/d(G)} \]
where $d(G) \in \Q$ is explicitly given, and $f(X) \asymp g(X)$ means that there exist positive constants $c_1,c_2$ such that $c_1 g(X) \leq f(X) \leq c_2 g(X)$.  In the case $G \simeq \Z/2\Z$, i.e., the case of 2-torsion, they
show the count is $cX^{1/2}+O(X^{1/3})$ for an explicit constant $c\approx 3.1969$ \cite[Theorem 5.5]{HS}.  (For weaker but related results, see also Duke \cite[Proof of Theorem 1]{Duke} and Grant \cite[Section~2]{Grant}.)

In this article, we count elliptic curves with a nontrivial cyclic isogeny defined over $\Q$.  An elliptic curve has a $2$-isogeny if and only if it has a $2$-torsion point, so the above result of Duke, Grant, and Harron--Snowden
handles this case.  The next interesting case concerns isogenies of degree $3$.  Our main result is as follows.

\begin{thm} \label{thm:mainthm}
Let $N_3(X)$ count the number of elliptic curves $E$ with $\hht(E) \leq X$ that possess a $3$-isogeny defined over $\Q$.  Then there exist constants $c_1,c_2$ such that
\[ N_3(X) = \frac{2}{3\sqrt{3}\zeta(6)}X^{1/2} + c_1X^{1/3}\log X + c_2 X^{1/3} + O(X^{7/24}). \]
Moreover, we have 
\[ c_1=\frac{c_0}{8\pi^2\zeta(4)}= 0.107437\ldots \]
where $c_0$ is an explicitly given integral \eqref{eq:c0eval}, and the constant $c_2$ is effectively computable.  
\end{thm}

We obtain the same asymptotic in Theorem \ref{thm:mainthm} if we instead count elliptic curves \emph{equipped} with a $3$-isogeny (that is, counting with multiplicity): see Proposition \ref{prop:samecount}.  Surprisingly, the main term of order $X^{1/2}$ counts \emph{just} those elliptic curves with $A=0$ and $j$-invariant equal to $0$ (having complex multiplication by the quadratic order of discriminant $-3$).  Theorem \ref{thm:mainthm} matches computations performed out to $X=10^{25}$---see section \ref{sec:computations}.  

The difficulty in computing the constant $c_2$ in the above theorem arises in applying a knotty batch of local conditions; our computations suggest that $c_2 \approx 0.16$.  If we count without these conditions, we find the explicit constant $c_6 =1.1204\dots$, given in \eqref{eqn:c6}---it is already quite complicated.

Theorem \ref{thm:mainthm} may be interpreted in alternative geometric language as follows.  Let $X_0(3)$ be the modular curve parametrizing (generalized) elliptic curves equipped with an isogeny of degree $3$.  
Then $N(X)$ counts rational points of bounded height on $X_0(3)$  with respect to the height arising from the pullback of the natural height on the $j$-line $X(1)$.  
 From this vantage point, the main term corresponds to a single elliptic point of order $3$ on $X_0(3)$!  The modular curves $X_0(N)$ are not fine moduli spaces (owing to quadratic twists), so our proof of Theorem \ref{thm:mainthm} is quite different than the method used by Harron--Snowden: in particular, a logarithmic term presents itself for the first time.  We hope that our method and the lower-order terms in our result will be useful in understanding counts of rational points on stacky curves more generally.

\subsection*{Contents} 
The paper is organized as follows.  We begin in section \ref{sec:setup} with a setup and exhibiting the main term, then in section \ref{sec:magnitude} as a warmup we prove the right order of magnitude for the secondary term.  In section \ref{sec:asymp}, we refine this approach to prove an asymptotic for the secondary term, and then we exhibit a tertiary term in section \ref{sec:secterm}.  We conclude in section \ref{sec:computations} with our computations.

\subsection*{Acknowledgments} 
The authors thank John Cullinan for helpful conversations.  Pizzo was supported by the Jack Byrne Scholars program at Dartmouth College.  Voight was supported by a Simons Collaboration grant (550029).

\section{Setup} \label{sec:setup}

In this section, we set up the problem in a manner suitable for direct investigation.  We continue the notation from the introduction.  

Let $\scrE$ denote the set of elliptic curves $E$ over $\Q$ in the form \eqref{eqn:Eyw2} (minimal, with nonzero discriminant).  For $X \in \R_{>0}$, let $\scrE_{\leq X} \colonequals \{E \in \scrE : \hht(E) \leq X\}$ be the set of elliptic curves $E$ over $\Q$ with height at most $X$.  We are interested in asymptotics for the functions
\begin{equation} \label{eqn:n3xp}
\begin{aligned}
N_3(X) \colonequals& \#\{E \in \scrE_{\leq X} : \text{$E$ has a $3$-isogeny defined over $\Q$}\}, \\
N_3'(X) \colonequals& \#\{(E,\phi) : E \in \scrE_{\leq X} \text{ and $\phi\colon E \to E'$ is a $3$-isogeny defined over $\Q$}\}.
\end{aligned}
\end{equation}

To that end, let $E=E_{A,B} \in \scrE$, with $A,B \in \Z$.  The $3$-division polynomial of $E$ \cite[Exercise 3.7]{Silverman} is equal to
\begin{equation}
\psi(x)=\psi_{A,B}(x) \colonequals 3x^4+6Ax^2+12Bx-A^2;
\end{equation}
the roots of $\psi(x)$ are the $x$-coordinates of nontrivial $3$-torsion points on $E$.

\begin{lem} \label{lem:psiroot}
The elliptic curve $E$ has a $3$-isogeny defined over $\Q$ if and only if $\psi(x)$ has a root $a \in \Q$.
\end{lem}

\begin{proof}
Let $\varphi \colon E \to E'$ be a $3$-isogeny defined over $\Q$.  Then $\ker \varphi=\{\infty,\pm P\}$ is stable under the absolute Galois group $\Gal_\Q$, so $\sigma(P)=\pm P$.  Thus, $\sigma(x(P))=x(P)$ for all $\sigma \in \Gal_\Q$ and hence $a=x(P) \in \Q$ is a root of $\psi(x)$ by definition.  Conversely, if $\psi(a)=0$ with $a \in \Q$, then letting $\pm P \colonequals (a,\pm\sqrt{f(a)})$ we obtain $C \colonequals \{\infty,\pm P\}$ a Galois stable subgroup of order $3$ and accordingly the map $\varphi \colon E \to E/C=E'$ is a $3$-isogeny defined over $\Q$: explicitly, by the formula of V\'elu we have 
\[ E' \colon y^2 = x^3 - 16(9A+30a^2)x - (9A^2/a + 114Aa + 253a^3) \]
(but such $E'$ is not necessarily in our designated form).  
\end{proof}

\begin{lem} \label{lem:ainZZ}
If $a \in \Q$ is a root of $\psi(x)$, then $a \in \Z$.
\end{lem}

\begin{proof}
For the first statement, by the rational root test, $a_0=3a \in \Z$.  Thus 
\begin{equation} 
0 = 27\psi(a)=a_0^4+18Aa_0^2+108Ba_0-27A^2 
\end{equation}
so $3 \mid a_0$, whence $a \in \Z$.  
\end{proof}

Although the polynomial $\psi(x)$ is irreducible in $\Z[A,B][x]$, the special case where $A=0$ gives $\psi_{0,B}(x)=3x(x^3+4B)$ and so $a=0$ is automatically a root, corresponding to the elliptic curve $E \colon y^2 = x^3 + B$ and the isogeny $\varphi \colon E \to E'$ by
\begin{equation} 
\varphi(x,y)=\left(\frac{x^3+4B}{x^2}, \frac{x^3-8B}{x^3}y\right)
\end{equation}
We count these easily.

\begin{lem} \label{lem:Aeq0cnt}
Let $N_3(X)_{A=0}$ 
be defined as in \eqref{eqn:n3xp} but restricted to $E \in \scrE_{\leq X}$ with $A=0$.  Then 
\[ N_3(X)_{A=0}=\frac{2}{3\sqrt{3}\zeta(6)}X^{1/2} + O(X^{1/6})
~\hbox{ and }~N_3'(X)_{A=0} = N_3(X)_{A=0}+O(X^{1/6}). \]
\end{lem}

\begin{proof}
In light of the above, we have 
\[ N_3(X)_{A=0}=\#\{B \in \Z : \text{$\abs{27B^2}\leq X$ and $\ell^6 \nmid B$ for any prime $\ell$}\}; \]
a standard sieve gives this count as $\displaystyle{\frac{2}{3\sqrt{3}\zeta(6)}}X^{1/2} + O(X^{1/12})$, see
Pappalardi \cite{Pappalardi}.  
If such an elliptic curve had another $3$-isogeny, corresponding to a root of $\psi(x)/x=x^3+4B$, then $-4B$ is a cube and the count of such is $O(X^{1/6})$. 
\end{proof}

With these lemmas in hand, we define our explicit counting function.  For $X>0$, let $N(X)$ denote the number of ordered triples $(A,B,a) \in \Z^3$ satisfying:
\begin{enumerate}
\item[(N1)] \label{N1}
$A\ne 0$ and $\psi_{A,B}(a)=0$;
\item[(N2)] \label{N2}
$\abs{4A^3}\leq X$ and $\abs{27B^2}\leq X$;
\item[(N3)] \label{N3}
$4A^3+27B^2\ne0$; and
\item[(N4)] \label{N4}
there is no prime $\ell$ with $\ell^4\mid A$ and $\ell^6\mid B$.  
\end{enumerate}
That is to say, we define
\begin{equation} \label{defn:NXdef}
N(X) \colonequals \#\{(A,B,a) \in \Z^3 : \text{all conditions (N1)--(N4) hold}\}.
\end{equation}

We have excluded from $N(X)$ the count for $A=0$ from the function $N(X)$; we have handled this in 
Lemma \ref{lem:Aeq0cnt}.

\begin{cor} \label{cor:N3p}
We have
\[ N_3'(X) = \frac{2}{3\sqrt{3}\zeta(6)}X^{1/2} + N(X) + O(X^{1/6}). \]
\end{cor}

\begin{proof}
This corollary is immediate from Lemmas \ref{lem:psiroot}, \ref{lem:ainZZ}, and \ref{lem:Aeq0cnt}.
\end{proof}

To conclude this section, we compare $N_3(X)$ and $N_3'(X)$.  

\begin{prop} \label{prop:samecount}
We have
\[ N_3'(X) = N_3(X) + O(X^{1/6}\log X) \]
\end{prop}

\begin{proof}
The difference $N_3'(X)-N_3(X)$ counts elliptic curves with more than one $3$-isogeny.  
Let $E$ be an elliptic curve with (at least) two $3$-isogenies $\varphi_i\colon E \to E_i'$ and let $\ker\varphi_i=\langle P_i\rangle$  for $i=1,2$.  Then $\langle P_1,P_2 \rangle = E[3]$, and so the image of $\Gal_\Q$ acting on $E[3]$ is a subgroup of the group of diagonal matrices in $\GL_2(\F_3)$.  This property is preserved by any twist of $E$, so such elliptic curves are characterized by the form of their $j$-invariant, explicitly \cite[Table 1, 3D${}^0$-3a]{SZ}
\begin{equation} 
j(t) = \left(\frac{t(t+6)(t^2-6t+36)}{(t-3)(t^2+3t+9)}\right)^3
\end{equation}
for $t \in \Q \setminus \{3\}$.  Computing an elliptic surface for this $j$-invariant, we conclude that every such $E$ is of the form $y^2=x^3+u^2A(t)x+u^3B(t)$ for some $t,u \in \Q$, where 
\begin{equation}
\begin{aligned}
A(t) &= -3t (t+6)(t^2-6t+36) = -3t^4-648t,  \\
B(t) &= 2(t^2-6t-18)(t^4+6t^3+54t^2-108t+324) = 2t^6-1080t^3-11664\, .
\end{aligned}
\end{equation}
Then by Harron--Snowden \cite[Proposition 4.1]{HS} (with $(r,s)=(4,6)$ so $m=1$ and $n=2$), the number of such elliptic curves is bounded above (and below) by a constant times $X^{1/6}\log X$, as claimed.  
\end{proof}

In light of the above, our main result will follow from an asymptotic for the easier function $N(X)$ defined in \eqref{defn:NXdef}, and so we proceed to study this function.

\section{Order of magnitude} \label{sec:magnitude}

In this section, we introduce new variables $u,v,w$ that will be useful in the sequel, and
provide an argument that shows the right order of magnitude.  This argument explains the provenance of the logarithmic term in a natural way and motivates our approach.
We recall \eqref{defn:NXdef}, the definition of $N(X)$.  

\begin{theorem}
\label{thm:order}
There exist $c_3,c_4,X_0 \in \R_{>0}$ such that for all $X\ge X_0$, we have
$$
c_3X^{1/3}\log X\le N(X)\le c_4X^{1/3}\log X.
$$
\end{theorem}

We begin with a few observations. First, if $A,B,a \in \Z$ and $\psi_{A,B}(a)=0$, then 
\begin{equation}
\label{eq:12B}
12B=\frac{A^2}a-6Aa-3a^3.
\end{equation}

\begin{lem} \label{lem:condBs}
Let $A,a \in \Z$ with $a \neq 0$.  Then $(A^2/a)-6Aa-3a^3 \in 12\Z$ if and only if all of the following conditions hold:
\begin{enumerate}
\item[(B1)] \label{B1} $a\mid A^2$ and $3\mid(A^2/a)$; 
\item[(B2)] \label{B2} $A,a$ have the same parity; and
\item[(B3)] \label{B3} If $A,a$ are both even, then $4 \mid (A^2/a)$.
\end{enumerate}
\end{lem}

\begin{proof}
The verification is straightforward.
\end{proof}

\begin{lem} \label{lem:Abaalpha}
Let $A,B,a \in \Z$ satisfy conditions \textup{(N1)}--\textup{(N2)}.  Then
\begin{equation}
\label{eq:ub}
|a|\ll X^{1/6}\quad\text{and} \quad A^2/|a|\ll X^{1/2}.
\end{equation}
\end{lem}

\begin{proof}
Let $\alpha \colonequals a/X^{1/6}$.  Since $|A|<4^{-1/3}X^{1/3}$, we have
$$
A^2/|a|<4^{-2/3}\abs{\alpha}^{-1}X^{1/2},~~|Aa|<4^{-1/3}\abs{\alpha} X^{1/2}.
$$
The inequality for $B$ and \eqref{eq:12B} imply that
$$
3|a|^3\le\frac4{3^{1/2}}X^{1/2}+6|Aa|+\frac{A^2}{|a|},
$$
so that 
\begin{equation} \label{eqn:th3sqp}
3\abs{\alpha}^3\le\frac4{3^{1/2}}+\frac{6\abs{\alpha}}{4^{1/3}}+\frac1{4^{2/3}\abs{\alpha}}.
\end{equation}
The inequality \eqref{eqn:th3sqp} fails for $\abs{\alpha}$ large---in fact, we have $\abs{\alpha}<11/8$---which
proves the first part of \eqref{eq:ub}.  To get the second part, note that the
first part  and condition (N2) imply that $|Aa|\ll X^{1/2}$.  And since
\eqref{eq:12B} implies that
$$
A^2/|a|\le 12|B|+6|Aa|+3|a|^3,
$$
we have $A^2/|a|\ll X^{1/2}$.
\end{proof}

\begin{proof}[Proof of Theorem \textup{\ref{thm:order}}]
We first prove the upper bound.  Every nonzero $a \in \Z$
can be written uniquely as $a=uv^2$, where $u \in \Z$ is squarefree and $v \in \Z_{>0}$.
Replacing $a=uv^2$, we see that $a\mid A^2$ if and only if $uv\mid A$.  Therefore $A=uvw$ with $w \in \Z$ arbitrary.  The inequalities in \eqref{eq:ub} imply that there exist $c_5,c_6>0$ such that
\begin{equation} \label{eqn:uvwineq}
\begin{aligned}
0<|u|v^2 \le c_5X^{1/6}~\hbox{ and }~ 
0<|u|w^2 \le c_6X^{1/2}.
\end{aligned}
\end{equation}
Thus, 
\begin{equation}
N(X) \leq \#\{(u,v,w) \in \Z^3 : \text{$u$ squarefree, $v>0$, and the inequalities \eqref{eqn:uvwineq} hold}\}. 
\end{equation}
For $X\ge2$, we have
\begin{equation}
\begin{aligned}
N(X)&\le\sum_{|u|v^2\le c_5X^{1/6}}\sum_{|u|w^2\le c_6X^{1/2}}1\ll \sum_{|u|v^2\le c_5X^{1/6}}\frac{X^{1/4}}{|u|^{1/2}}\\
&\le X^{1/4}\sum_{0<v\le c_5^{1/2}X^{1/12}}\,\sum_{|u|\le c_5X^{1/6}/v^2}\frac1{|u|^{1/2}}\\
&\ll X^{1/3}\sum_{0<v\le c_5^{1/2}X^{1/12}}\frac1v\ll X^{1/3}\log X.
\end{aligned}
\end{equation}

For the lower bound, we let $u,v,w$ range over positive, odd, squarefree numbers with $3 \mid w$ and let $a=uv^2$ and $A=uvw$ as in the previous paragraph; these ensure that conditions (B1)--(B3) hold, so by Lemma \ref{lem:condBs} we have $B \in \Z$.  Conditions \hyperref[N1]{(N1)} and \hyperref[N4]{(N4)} are also satisfied, and condition \hyperref[N3]{(N3)} is negligible.  
To ensure \hyperref[N2]{(N2)}, we choose 
\begin{equation}
\label{eq:3ineq}
v\le X^{1/24}, ~~uv^2<\frac12X^{1/6},~~w<uv^3.
\end{equation}
Then $A=uvw<u^2v^4<\frac14X^{1/3}$ so $\abs{4A^3} \leq X$.  Moreover,   
\begin{equation}
\begin{aligned}
-12B &= 3u^3v^6+6u^2v^3w-uw^2=3u^3v^6\left(1+2\frac{w}{uv^3}-\frac13\left(\frac{w}{uv^3}\right)^2\right) \\
&< 3\left(\frac{1}{2}X^{1/6}\right)^3 \frac83 = X^{1/2} 
\end{aligned}
\end{equation}
since $0<w/uv^3 \leq 1$ and the polynomial $1+2t-\frac13t^2$ on $[0,1]$ is positive and takes the maximum value $\frac83$.   Thus, all conditions are satisfied.  

We now count the choices for $u,v,w$ with the above conditions: we have
\begin{equation}
\begin{aligned}
N(X) &\ge \sum_{v\le X^{1/24}}\sum_{uv^2<\frac14X^{1/6}}\sum_{w<uv^3}1 
&\gg \sum_{v\le X^{1/24}}\sum_{u<\frac14X^{1/6}/v^2}uv^3.
\end{aligned}
\end{equation}
By partial summation, the inner sum on $u$ is $\gg X^{1/3}/v$, and then another partial summation
gives that $N_0(X)\gg X^{1/3}\log X$, which completes the proof of the lower bound.
\end{proof}

\section{An asymptotic} \label{sec:asymp}

In this section, we prove an asymptotic for $N(X)$.  

We recall some notation introduced in the proof of Theorem \ref{thm:order}.
Let $(A,B,a) \in \Z^3$ satisfy \hyperref[N1]{(N1)}, so $a \neq 0$ and $B$ is determined by $A,a$ as in Lemma \ref{lem:condBs}.  Write
\begin{equation} \label{eqn:aA}
\begin{aligned}
a &=uv^2 \\
A &=uvw
\end{aligned}
\end{equation}
with $u \in \Z$ squarefree, $v \in \Z_{>0}$, and $w \in \Z_{\ne0}$.  Then
\begin{equation}  \label{eqn:twelveBinuvw}
12B=uw^2-6u^2v^3w-3u^3v^6. 
\end{equation}
We rewrite condition \hyperref[N4]{(N4)} and the conditions in Lemma \ref{lem:condBs} in terms of the quantities $u,v,w$ as follows.

\begin{lem} \label{lem:uvw12z}
Conditions \textup{\hyperref[B1]{(B1)}--\hyperref[B3]{(B3)}} and \textup{\hyperref[N4]{(N4)}} hold
if and only if all of the following conditions hold:
\begin{enumerate}
    \item[(W1)] \label{W1} $uv\equiv w\pmod 2$;
    \item[(W2)] \label{W2} Not both $2^2\mid v$ and $2^4\mid w$ occur;
    \item[(W3)] \label{W3} Not all of $2\nmid u$, $2\,\|\,v$, and $2^3\,\|\,w$ occur;
    \item[(W4)] \label{W4} Not all of $2\mid u$, $2\,\|\,v$, and $2^4\mid w$ occur;
    \item[(W5)] \label{W5} $3\mid uw$;
    \item[(W6)] \label{W6} Not both $3\mid v$ and $3^4\mid uw$ occur; and
    \item[(W7)] \label{W7} For each prime $\ell>3$, not both $\ell\mid v$ and $\ell^3\mid w$ occur.
\end{enumerate}
\end{lem}

\begin{proof}
This lemma can be proven by a tedious case-by-case analysis.  Alternatively, the conditions \hyperref[B1]{(B1)}--\hyperref[B3]{(B3)} are determined by congruence conditions modulo $16$ and $81$, so we may also just loop over the possibilities by computer.
\end{proof}

\begin{lem} \label{lem:localconst}
The proportion among $(u,v,w)$ (with $u$ squarefree) satisfying the conditions \textup{\hyperref[W1]{(W1)}}--\textup{\hyperref[W7]{(W7)}} is $(4\zeta(4))^{-1}$.  
\end{lem}

\begin{proof}
For the conditions \hyperref[W1]{(W1)}--\hyperref[W6]{(W6)}, we just count residue classes (as in Lemma \ref{lem:uvw12z}): we find proportions $15/32$ for conditions \hyperref[W1]{(W1)}--\hyperref[W4]{(W4)} and $40/81$ for \hyperref[W5]{(W5)}--\hyperref[W6]{(W6)}.  For condition \hyperref[W7]{(W7)}, the proportion of cases where $\ell\mid v$ and $\ell^3\mid w$ is $1/\ell^4$, thus the correction factor is
$$
\prod_{\ell>3}\left(1-\frac1{\ell^4}\right)=\frac{16}{15}\cdot\frac{81}{80}\cdot\frac1{\zeta(4)} = \frac{27}{25\zeta(4)}.
$$
Thus the total proportion is 
\[ 
\frac{15}{32} \cdot \frac{40}{81} \cdot \frac{27}{25\zeta(4)} = \frac{1}{4\zeta(4)}. \qedhere \]
\end{proof}

Let $X>0$, and suppose $(A,B,a)$ is counted by $N(X)$.  Define $\alpha,\beta \in \R_{>0}$ by 
\begin{equation} \label{eqn:auvwxalpha}
\begin{aligned}
a = uv^2 &= \alpha X^{1/6}, \\
w &=\beta uv^3.
\end{aligned}
\end{equation}
(The quantity $\alpha$ arose in the proof of Lemma \ref{lem:Abaalpha}.)  
Moreover, define the functions
\begin{equation}
\begin{aligned}
f(\beta) &\colonequals \frac1{2^{1/3}|\beta|^{1/2}}, \\
g(\beta) &\colonequals \frac{4^{1/3}}{3^{1/2}|1+2\beta-\frac13\beta^2|^{1/3}}, \\
h(\beta) &\colonequals \min\{f(\beta),g(\beta)\}.
\end{aligned}
\end{equation}

\begin{center}
\begin{overpic}[width=0.8\textwidth]{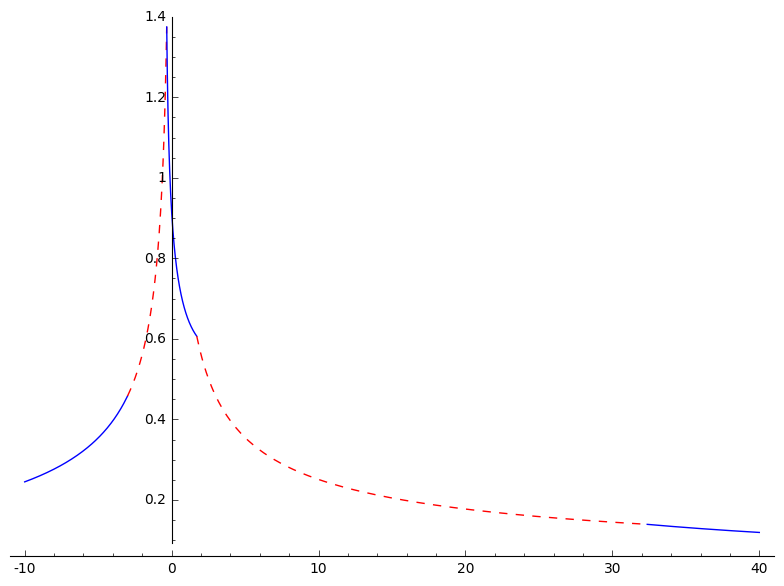}
 \put (23,70){$h(\beta)$}
 \put (100,3){$\beta$}
 \put (15,0){$\beta_4$} \put (16,3){$\shortmid$}
 \put (19.5,0){$\beta_3$} \put (20.5,3){$\shortmid$}
 \put (24,0){$\beta_2$} \put (25,3){$\shortmid$}
 \put (81,0){$\beta_1$} \put (82,3){$\shortmid$}
\end{overpic}
\end{center}
The transition points for the piecewise function $h(\beta)$ occur at
\begin{equation} 
\beta_1 \colonequals 32.37198796\dots,~~ \beta_2 \colonequals 1.71119188\dots,~~\beta_3 \colonequals -\frac13,~~\beta_4 \colonequals -3;
\end{equation}
the transition points $\beta_1,\beta_2$ are algebraic numbers.  Then $h(\beta)=g(\beta)$ on the intervals $(-\infty,\beta_4)$, $(\beta_3,\beta_2)$, and $(\beta_1,\infty)$
 and $h(\beta)=f(\beta)$ on the complementary intervals $(\beta_4,\beta_3)$ and $(\beta_2,\beta_1)$.  

We compute numerically that
\begin{equation}
\label{eq:c0eval}
    c_0 \colonequals \int_{-\infty}^{\infty}h(\beta)^2\,d\beta=9.1812458638\dots.
\end{equation}

The relevance of these quantities (as well as their weighting) is made plain by the following lemma.
 
\begin{lem} \label{lem:alphabetah}
The triple $(u,v,w)$ satisfies \textup{\hyperref[N2]{(N2)}} if and only if 
\[ \abs{\alpha} \leq h(\beta). \]
\end{lem}

\begin{proof}
Since $A=uvw=\beta u^2v^4=\alpha^2\beta X^{1/3}$, the first inequality in \hyperref[N2]{(N2)} is equivalent to
\begin{equation}
    \label{eq:ineq1}
|\alpha^2\beta| \leq 4^{-1/3}.
\end{equation}
In addition, we have
\[ -12B=3u^3v^6\left(1+2w/uv^3-\frac13\left(\frac{w}{uv^3}\right)^2\right)=3\alpha^3X^{1/2}\left(1+2\beta-\frac13\beta^2\right), \] 
so that the second inequality in \hyperref[N2]{(N2)} is equivalent to
\begin{equation}
    \label{eq:ineq2}
 \big|\alpha^3\big(1+2\beta-\frac13\beta^2\big)\big| \leq \frac4{3^{3/2}}.   
\end{equation}
The result then follows from \eqref{eq:ineq1} and \eqref{eq:ineq2}.
\end{proof} 

We then have the following first version of our main result.

\begin{thm} \label{thm:NXasymp}
We have
\[ N(X) \sim c_1 X^{1/3} \log X \]
where
\[ c_1 \colonequals \frac{c_0}{8\pi^2\zeta(4)}= 0.10743725502\ldots \]
and $c_0$ is defined in \eqref{eq:c0eval}.
\end{thm}

\begin{proof}
Via \eqref{eqn:aA}--\eqref{eqn:twelveBinuvw}, $N(X)$ counts $(u,v,w) \in \Z^3$ with $u$ squarefree, $v$ positive, $w \neq 0$, such that conditions \hyperref[N2]{(N2)}--\hyperref[N3]{(N3)} hold as well as the local conditions \hyperref[W1]{(W1)}--\hyperref[W7]{(W7)} (which implies \hyperref[N4]{(N4)}).  We may ignore condition \hyperref[N3]{(N3)} as negligible: for each choice of $u,v$ there are $O(1)$ choices of $w$ where \hyperref[N3]{(N3)} fails, subtracting at most $O(X^{1/6})$ from the count.  

We first show how to count triples $u,v,w$ satisfying \hyperref[N2]{(N2)}, not necessarily the local conditions, and define
\begin{equation} \label{eqn:N0X}
N_0(X) \colonequals \#\{(u,v,w) \in \Z^3 : \text{$u$ squarefree, $v>0$, and \hyperref[N2]{(N2)} holds}\}. 
\end{equation}
We suppress the reminder that $u$ is taken to be squarefree.  The number of triples with $w=0$ is negligible, so we ignore this condition.

Let $X>0$.  For $(u,v,w)$ counted by $N_0(X)$, we organize by the value of $\beta=w/uv^3 \in \Q$.  Taking $\beta$ in an interval $I$ of length $s$ that does not contain a transition point in its interior, the integers $u,v$ are constrained by 
\[ \abs{a}=\abs{u}v^2<\abs{\alpha}X^{1/6} < h(\beta)X^{1/6} \] 
(with $h(\beta)$ minimal on $I$, taking left or right endpoint) by Lemma \ref{lem:alphabetah}.  Given $u,v$, we have $w=\beta uv^3 \in uv^3 I$ giving approximately $uv^3$ possible values of $w$.  Repeating this argument with Riemann sum estimates, we obtain
\begin{equation} \label{eqn:xn0x}
N_0(X) \sim \int_{-\infty}^{\infty}\sum_{\substack{|u|v^2<h(\beta)X^{1/6}\\v>0}}|u|v^3\,\dD\beta
\end{equation}
as $X \to \infty$.  (For a more refined approach with an error term, see \eqref{eqn:N0Xestimateerror} below.)

We now evaluate this integral.  Recall that
\[ \sum_{\abs{u} \leq t} \abs{u} \sim \frac{6}{\pi^2} t^2; \]
inputting this into \eqref{eqn:xn0x} and letting $X \to \infty$, we obtain
\begin{equation}
\begin{aligned}
&\int_{-\infty}^{\infty}\sum_{\substack{v^2<h(\beta)X^{1/6} \\ v>0}} v^3 \sum_{|u|<h(\beta)X^{1/6}/v^2} |u|\,\dD\beta \sim 
\frac6{\pi^2}\int_{-\infty}^{\infty}\sum_{v<h(\beta)^{1/2}X^{1/12}}v^3 \frac{h(\beta)^2X^{1/3}}{v^4}\,\dD\beta \\
&\qquad\qquad\qquad \sim \frac{6X^{1/3}}{\pi^2}\int_{-\infty}^{\infty}h(\beta)^2\int_{1}^{h(\beta)^{1/2}X^{1/12}} \frac{1}{v}\,\dD{v}\,\dD\beta \\
&\qquad\qquad\qquad = \frac{6X^{1/3}}{\pi^2}\int_{-\infty}^{\infty}h(\beta)^2\log(h(\beta)^{1/2}X^{1/12})\,\dD\beta \\
&\qquad\qquad\qquad \sim \frac1{2\pi^2}X^{1/3}\log X\int_{-\infty}^{\infty}h(\beta)^2\,\dD\beta = \frac{c_0}{2\pi^2} X^{1/3} \log X.
\end{aligned}
\end{equation}
%

Finally, we impose the local constraints \hyperref[W1]{(W1)}--\hyperref[W7]{(W7)}.  The first 6 of these are clear.  
To impose \hyperref[W7]{(W7)} note that
\[
\frac{27}{25\zeta(4)}=\prod_{\ell>3}\left(1-\frac1{\ell^4}\right)=\sum_{\gcd(d,6)=1}\frac{\mu(d)}{d^4}.
\]
The sum converges rapidly, in fact, for $Z>1$,
\[
\Bigg|\frac{27}{25\zeta(4)}-\sum_{\substack{\gcd(d,6)=1\\d\le Z}}\frac{\mu(d)}{d^4}\Bigg|\ll \frac1{Z^3}.
\]
Further, the proportion of triples $u,v,w$ with $d\mid v$ and $d^3\mid w$ for some $d>Z$ tends to 0
as $Z\to\infty$.  So, imposing \hyperref[W7]{(W7)} introduces the factor $27/(25\zeta(4))$ as in 
 Lemma \ref{lem:localconst}.  We conclude that
\[ N(X) \sim \frac{1}{4\zeta(4)}N_0(X) \sim c_1 X^{1/3}\log X \]
as $X \to \infty$, as claimed.
\end{proof}

\section{Secondary term} \label{sec:secterm}

In this section, we work on a secondary term for $N(X)$ (giving a tertiary term for $N_3(X)$).  

We start by explaining how this works for the function $N_0(X)$ defined in \eqref{eqn:N0X}, namely, the triples $(u,v,w) \in \Z^3$ such that $u$ is squarefree, $v>0$, and $\abs{\alpha} \leq h(\beta)$ where $\alpha,\beta$ are defined by \eqref{eqn:auvwxalpha}.  We discuss the modifications to this approach for $N(X)$ below.  

We begin by working out an analog of Euler's constant for the squarefree harmonic series.

\begin{lemma}
 \label{lem:harm}
For real numbers $x\ge1$ we have
$$
\sum_{\substack{0<u\le x\\\text{$u$ squarefree}}}\frac1u=\frac1{\zeta(2)}\log x+\gamma_0+O(x^{-1/2}\log x)=1.0438945157\dots,
$$
where 
\begin{equation} \label{defn:gamma0} 
\gamma_0 \colonequals \frac{\gamma\zeta(2)-2\zeta'(2)}{\zeta(2)^2} 
\end{equation}
and $\gamma$ is Euler's constant.
\end{lemma}
\begin{proof}
The integer variables $u,v,d$ in this proof are positive.  
We have
$$
\sum_{\substack{u\le x\\u~{\rm squarefree}}}\frac1u=\sum_{u\le x}\sum_{d^2\,|\,u}\frac{\mu(d)}u
=\sum_{d\le x^{1/2}}\frac{\mu(d)}{d^2}\sum_{v\le x/d^2}\frac1v
=\sum_{d\le x^{1/2}}\frac{\mu(d)}{d^2}\left(\log\Big(\frac x{d^2}\Big)+\gamma+O\Big(\frac{d^2}x\Big)\right).
$$
The $O$-terms add to $O(x^{-1/2})$.  Since
$$
\sum_{d\le x^{1/2}}\frac{\mu(d)}{d^2}=\sum_{d}\frac{\mu(d)}{d^2}-\sum_{d> x^{1/2}}\frac{\mu(d)}{d^2}
=\frac1{\zeta(2)}+O(x^{-1/2})
$$
and 
$$
\sum_{d\le x^{1/2}}\frac{2\mu(d)\log d}{d^2}=
\sum_{d}\frac{2\mu(d)\log d}{d^2}-\sum_{d> x^{1/2}}\frac{2\mu(d)\log d}{d^2}=\frac{2\zeta'(2)}{\zeta(2)^2}+O\left(x^{-1/2}\log x\right),
$$
the result follows.
\end{proof}

\begin{thm} \label{thm:N0Xsecond}
We have
\[ N_0(X) = \frac{c_0}{2\pi^2}X^{1/3}\log X+c_6X^{1/3}+O(X^{7/24}) \]
where $c_0$ is defined in \eqref{eq:c0eval} and
\begin{equation} \label{eqn:c6}
c_6 \colonequals \left(\frac{\gamma_0}2+\frac{6\gamma}{\pi^2}-\frac3{2\pi^2}\right)c_0 
+ \frac{3}{\pi^2}\int_{-\infty}^\infty h(\beta)^2\log h(\beta)\,\dD{\beta} = 1.12042819875\dots 
\end{equation}
where $\gamma_0$ is defined in \eqref{defn:gamma0} and $\gamma$ is Euler's constant.
\end{thm}

\begin{proof}
We return to the derivation of the integral expression \eqref{eqn:xn0x} and  consider the contribution of a single term $a=uv^2$.  With $\alpha=a/X^{1/6}$,
the contribution of $a$ to the integral is
\begin{equation} \label{eqn:ddbetauv}
\int_{h(\beta)\ge|\alpha|}|u|v^3\,\dD\beta=|u|v^3\int_{h(\beta)\ge|\alpha|}\,\dD\beta.
\end{equation}

 Note that $h$ is continuous.
 Let $h_1 \colonequals h|_{(-\infty,-1/3]}$ and $h_2 \colonequals h|_{[-1/3,\infty)}$.  Then $h_1$ is strictly increasing and $h_2$ is strictly decreasing.  Letting $j_1,j_2$ be the inverses of $h_1,h_2$, respectively, we have for any $t\in(0,h(-1/3)]$ that 
 \begin{equation}
     \label{eq:j1j2}
 \{\beta\in\RR:h(\beta)\ge t\}=[j_2(t),j_1(t)].
  \end{equation}
Plugging \eqref{eq:j1j2} into the integral \eqref{eqn:ddbetauv}, we obtain $j_1(|\alpha|)-j_2(|\alpha|)$.  

For a choice of $a=uv^2$, we count the number of nonzero integers $w$ with $w/(\abs{u}v^3) \in[j_2(|\alpha|),j_1(|\alpha|)]$: this is equal to
\[ \abs{u}v^3(j_1(|\alpha|)-j_2(|\alpha|))+O(1). \]  
So, the error when considering the integral in \eqref{eqn:xn0x} is $O(X^{1/6})$, i.e.,
\begin{equation} \label{eqn:N0Xestimateerror}
N_0(X) = 
\int_{-\infty}^\infty\sum_{\substack{|u|v^2\le h(\beta)X^{1/6}\\v>0}}|u|v^3\,\dD\beta+O(X^{1/6}).
\end{equation}

We next consider the evaluation of the integrand 
\begin{equation}
S \colonequals \sum_{\substack{|u|v^2<h(\beta)X^{1/6}\\v>0}}|u|v^3
\end{equation}
(with the continued understanding that $u$ is squarefree).  
Let $H(\beta) \colonequals h(\beta)^{1/4}X^{1/24}$, so that if $\abs{u} v^2\le h(\beta)X^{1/6}$,
then either $\abs{u} \le H^2$ or $v\le H$.
Let $S_1$ be the contribution to the integrand
when $|u|\le H^2$, let $S_2$ be the contribution when $v\le H$, and let $S_3$ be the contribution
when both $|u|\le H^2$ and $v\le H$.  Then 
$$
S=S_1+S_2-S_3.
$$

Using that $\sum_{0<v\le t}v^3=\frac14t^4+O(t^3)$,
for a given value of $u$ with $|u|\le H^2$, we have
$$
\sum_{v\le h(\beta)^{1/2}X^{1/12}/|u|^{1/2}}|u|v^3=\frac14|u|\left(\frac{h(\beta)^2X^{1/3}}{|u|^2}+O\Big(\frac{h(\beta)^{3/2}X^{1/4}}{|u|^{3/2}}\Big)\right).
$$
Summing this over squarefree numbers $u$ with $|u|\le H^2$ and using Lemma \ref{lem:harm}, we get
\begin{equation} \label{eqn:S1}
\begin{aligned}
S_1=&\frac14h(\beta)^2X^{1/3}\cdot2\left(\frac6{\pi^2}\log H^2+\gamma_0\right)+O\big(h(\beta)^2X^{1/3}H^{-2}\log H+h(\beta)^{3/2}X^{1/4}H\big)\\
=&\frac1{4\pi^2}h(\beta)^2X^{1/3}\log X+h(\beta)^2\left(\frac12\gamma_0+\frac3{2\pi^2}\log h(\beta)\right)X^{1/3}\\
&\qquad+O\big(h(\beta)^{3/2}X^{1/4}\log X\big)+O\big(h(\beta)^{7/4}X^{7/24}\big).
\end{aligned}
\end{equation}

Next we consider $S_2$.  For a given value of $v\le H$, we have
\begin{equation}
\sum_{\substack{|u|\le h(\beta)X^{1/6}/v^2}}|u|v^3
=2\cdot\frac12\cdot\frac6{\pi^2}h(\beta)^2X^{1/3}v^{-1}+O\Big(h(\beta)^{3/2}X^{1/4}\Big),
\end{equation}
using that the number of squarefree numbers up to a bound $x$ is $\frac6{\pi^2}x+O(x^{1/2})$ and partial summation.
Summing for $v\le H$ we get
\begin{equation} \label{eqn:S2}
\begin{aligned}
    S_2&=\frac6{\pi^2}h(\beta)^2X^{1/3}\Big(\frac1{24}\log X+\gamma+\frac14\log h(\beta)+O(1/H)\Big)
    +O(h(\beta)^{3/2}X^{1/4}H)\\
    &=\frac1{4\pi^2}h(\beta)^2X^{1/3}\log X+\frac6{\pi^2}h(\beta)^2\Big(\gamma+\frac14\log h(\beta)\Big)X^{1/3}
    +O(h(\beta)^{7/4}X^{7/24}).
\end{aligned}
\end{equation}

Finally, for $S_3$ we have
\begin{equation} \label{eqn:S3}
\begin{aligned}
S_3&=\left(\frac6{\pi^2}H^4+O(H^{3})\right)
\left(\frac14H^4+O(H^{3})\right)=\frac3{2\pi^2}H^8+O(H^{7})\\
&=\frac3{2\pi^2}h(\beta)^2X^{1/3}+O(h(\beta)^{7/4}X^{7/24}).
\end{aligned}
\end{equation}

Since $S=S_1+S_2-S_3$, combining \eqref{eqn:S3}, \eqref{eqn:S2}, and \eqref{eqn:S3} we obtain
\begin{equation} \label{eqn:S}
\begin{aligned}
S=&\frac{h(\beta)^2}{2\pi^2}X^{1/3}\log X
+h(\beta)^2\Big(\frac{\gamma_0}2+\frac{6\gamma}{\pi^2}+\frac3{\pi^2}\log h(\beta)-\frac3{2\pi^2}\Big)X^{1/3}\\
&\qquad+O\big(h(\beta)^{3/2}X^{1/4}\log X\big)+O\big(h(\beta)^{7/4}X^{7/24}\big).
\end{aligned}
\end{equation}

The expression \eqref{eqn:S} is then to be integrated over all $\beta$ to obtain $N_0(X)$ as in \eqref{eqn:N0Xestimateerror}.  However, in this integration, we may suppose
that $|\beta|\ll X^{1/4}$, since $h(\beta)\asymp |\beta|^{-2/3}$ and we may suppose that $h(\beta)X^{1/6}\ge1$.
Thus, integrating the first error term gives $O(X^{1/4}(\log X)^2)$ and integrating the second error
term gives $O(X^{7/24})$.
We conclude that 
\begin{equation}
\label{eq:sec}
\begin{aligned}
\int_{-\infty}^\infty\sum_{\substack{|u|v^2\le h(\beta)X^{1/6}\\v>0}}|u|v^3\,\dD\beta
&= 
\frac{c_0}{2\pi^2}X^{1/3}\log X+\Big(\frac{\gamma_0}2+\frac{6\gamma}{\pi^2}-\frac3{2\pi^2}\Big)c_0X^{1/3}\\
&\qquad+
\frac3{\pi^2}X^{1/3}\int_{-\infty}^\infty h(\beta)^2\log h(\beta)\,\dD\beta+O\big(X^{7/24}\big).
\end{aligned}
\end{equation}

We compute numerically that
\begin{equation}
\int_{-\infty}^\infty h(\beta)^2\log h(\beta)\,d\beta=-18.0878968694\dots
\end{equation}
and so the coefficient of the secondary term of $N_0(X)$ is $c_6 = 1.12042819875\dots$.
\end{proof}

Before proving our main theorem, we prove one lemma, generalizing Lemma \ref{lem:harm}.  For $i\mid 6$ with $i>0$, let
\begin{equation}
H_i(x) \colonequals \sum_{\substack{0<u\le x\\u~{\rm squarefree}\\\gcd(u,6)=i}}\frac1u.
\end{equation}

\begin{lem} \label{lem:harm2}
We have
\[
H_1(x)=\frac1{2\zeta(2)}\log x+\gamma_1+O\Big(\frac{\log x}{x^{1/2}}\Big)
\]
where 
\[
\gamma_1=\frac{\log432}{24\zeta(2)}+\frac{\gamma}{2\zeta(2)}-\frac{\zeta'(2)}{\zeta(2)^2},
\]
and $\gamma$ is Euler's constant.  Moreover, $H_i(x)=\frac1iH_1(\frac xi)$ for $i \mid 6$.
\end{lem}

\begin{proof}
The proof follows the same lines as Lemma \ref{lem:harm}.  
\end{proof}

We now prove our main result.

\begin{proof}[Proof of Theorem \textup{\ref{thm:mainthm}}]
The asymptotic for $N(X)$ was proven in Theorem \ref{thm:NXasymp} and a secondary term with power-saving error term for $N_0(X)$ was proven in Theorem \ref{thm:N0Xsecond}.  To finish, we claim that the local conditions \hyperref[W1]{(W1)}--\hyperref[W7]{(W7)} that move us from $N_0(X)$ to $N(X)$ can be applied in the course of the argument for Theorem \ref{thm:N0Xsecond} to obtain an (effectively computable) constant.  

Let $i,j,k,d \in \Z_{>0}$ satisfy: $i\mid 6$, $d$ squarefree and coprime to 6, $j\mid 12$, and $k\mid 6^4$.
Let $N_{i,j,k,d}(X)$ denote
the number of triples $u,v,w$ counted by $N_0(X)$ with $\gcd(u,6)=i$, $jd\mid v$, and $kd^3\mid w$.
Then with $i,j,k$ running over triples consistent with conditions \hyperref[W1]{(W1)}--\hyperref[W6]{(W6)}, a signed sum of the counts
$N_{i,j,k,d}(X)$ gives $N(X)$.  For example, take the case of $uvw$ coprime to 6, which satisfies
\hyperref[W1]{(W1)}--\hyperref[W6]{(W6)}.  The contribution of these triples to $N(X)$ is
\[
\sum_{j\mid 6}\sum_{k\mid 6}\sum_{\gcd(d,6)=1}\mu(j)\mu(k)\mu(d)N_{1,j,k,d}(X).
\]
We have similar expressions for other portions of the $u,v,w$-domain of triples.

We now estimate $N_{i,j,k,d}$ and control the contribution to $N(X)$ from large $d$.  For the latter, since $\abs{vw}\le A\ll X^{1/3}$, we have
$d\ll X^{1/12}$; so we may suppose that $d$ is so bounded.  Getting a good estimate for $N_{i,j,k,d}$ follows
in exactly the same way as with $N_0$.  In particular, we have the analogue of \eqref{eqn:N0Xestimateerror}:
\begin{equation}
N_{i,j,k,d}(X)=\int_{-\infty}^\infty\sum_{\substack{|u|v^2\le h(\beta)X^{1/6}\\\gcd(u,6)=i \\ jd\,|\,v}}\frac{|u|v^3}{kd^3}\,\dD\beta +O(X^{1/6}),
\end{equation}
where it is understood that $u$ is squarefree and $v>0$.
The sum here is estimated in the same way, by first considering the contribution when $|u|\le H^2$,
where $H=h(\beta)^{1/4}X^{1/24}$, then the contribution when $v\le H$, and finally the contribution
when both $|u|\le H^2$ and $v\le H$.  To accomplish this, we use the following asymptotic estimates:
\begin{equation}
\begin{aligned}
 \sum_{\substack{0<v\le x\\jd\,|\,v}}v^3  &=\frac14\frac{x^4}{jd}+O(x^3),\\ 
  \sum_{\substack{0<v\le x\\jd\,|\,v}}\frac1v&=\frac1{jd}\log x+\frac{\gamma-\log(jd)}{jd}+O\Big(\frac1{x^{1/2}d^{1/2}}\Big),\\
  \sum_{\substack{|u|~{\rm squarefree}\\|u|\le x\\\gcd(u,6)=i}}|u|&=\frac1{i\zeta(2)}x^2+O(x^{3/2}).
\end{aligned}
\end{equation}
We also need the sum of $1/|u|$, accomplished in Lemma \ref{lem:harm2}.

Putting these ingredients together, we get that
\begin{equation}
N_{i,j,k,d}(X)=\frac{c_{i,j,k}}{d^4}X^{1/3}\log X+\frac{c'_{i,j,k}}{d^4}X^{1/3}+O\Big(\frac{X^{7/24}}{d^3}\Big),
\end{equation}
where $c_{i,j,k},c'_{i,j,k}=O(1)$ uniformly, and summing these contribution gives the result.  
\end{proof}

\section{Computations} \label{sec:computations}

We conclude with some computations that give numerical verification of our asymptotic expression.  

We computed the functions $N_0(X)$ and $N(X)$ as follows.  First, we restrict to $u>0$ (still squarefree), since this gives exactly half the count.  Second, we loop over $u$ up to $\lfloor \frac{11}{8}X^{1/6} \rfloor$ (valid as in the proof of Lemma \ref{lem:Abaalpha}) and keep only squarefree $u$.  Then we loop over $v$ from $0$ up to $\lfloor \sqrt{\tfrac{11}{8}X^{1/6}/u} \rfloor$.  This gives us the value of $a=uv^2$.  Then plugging into $h$ gives
\begin{equation} 
\beta_{\textup{max}} \leq \max\left(\left\{\frac{X^{1/3}}{4^{1/3}a^2}, 3 + \sqrt{12+\frac{4}{\sqrt{3}}\frac{X^{1/2}}{a^3}}\right\}\right). 
\end{equation}
Then we loop over $w$ from $-\beta_{\textup{max}}uv^3$ to $\beta_{\textup{max}}uv^3$, ignoring $w=0$, and we take $A=uvw$.  We then check that $\abs{4A^3} \leq X$; and letting 
\[ B=\frac{1}{12}\left(\frac{A^2}{a} - 6Aa - 3a^3\right) \]
we check that $\abs{27B^2} \leq X$, and if so add to the count for $N_0(X)$.  For $N(X)$, we further check the local conditions \hyperref[B1]{(B1)}--\hyperref[B3]{(B3)} and \hyperref[N4]{(N4)} (or, equivalently, \hyperref[W1]{(W1)}--\hyperref[W7]{(W7)}).

In this manner, we thereby compute the data in Table \ref{tab2} for $X=10^m$: we computed $N_0(10^m)$ for $m \leq 20$ and $N(10^m)$ for $m \leq 25$.

A best fit with
\[ N_0(X) \stackrel?= c_7X^{1/3}\log(X) + c_8X^{1/3} \]
confirms
\begin{equation}
\begin{aligned}
c_7 =0.46527 &\approx 0.46513\ldots = \frac{c_0}{2\pi^2} \\
c_8 =1.1121 &\approx 1.1204\ldots = c_6
\end{aligned}
\end{equation}
and the difference between the first two columns is indeed small.  Similarly, a best fit with
\[ N(X) \stackrel?= c_9X^{1/3}\log(X) + c_{10}X^{1/3} \]
gives
\begin{equation}
\begin{aligned}
c_9 = 0.107400 &\approx 0.107437\ldots = c_1 \\
c_{10} = 0.165612 &\approx c_2 
\end{aligned}
\end{equation}
confirms the asymptotic, with an approximate value for the constant $c_2 \approx 0.16$ as also indicated in the fourth column.
  
  \begin{equation} \label{tab2}\addtocounter{equation}{1} \notag
\begin{gathered}
{\renewcommand{\arraystretch}{1.1}
\begin{tabular}{c||c|c||c|c} 
$m$ & $N_0(X)$ & $\displaystyle{\frac{c_0}{2\pi^2}}X^{1/3}\log(X) + c_6 X^{1/3}$ & $N(X)$ & $\displaystyle{\frac{N(X)-c_1X^{1/3}\log X}{X^{1/3}}}$\\[0.1in]
\hline
\hline
3 & 40 & 43 & 2 & -0.54215 \\
4 & 106 & 116 & 16 & -0.24688 \\
5 & 292 & 301 & 54 & -0.07352 \\
6 & 728 & 755 & 144 & -0.04430 \\
\vdots & \vdots & \vdots &\vdots &\vdots\\
16 & 3931970 & 3933212 & 887334 & 0.16051 \\
17 & 8968482 & 8970960 & 2026332 & 0.16008 \\
18 & 20396372 & 20398344 & 4615666 & 0.16276 \\
19 & 46250606 & 46254289 & 10476028 & 0.16226 \\
20 & 104614810 & 104622964 & 23720904 & 0.16285 \\
21 & 236105316 & 236113295 & 53583854 & 0.16333 \\
22 & 531764374 & 531764568 & 120772894 & 0.16335 \\
23 & 1195334414 & 1195363230 & 271694240 & 0.16366 \\
24 & 2682372754 & 2682431541 & 610085848 & 0.16366 \\
25 & 6009687100 & 6009862508 & 1367646478 & 0.16347 \\
\end{tabular}} \\
\text{Table \ref{tab2}: Data before and after applying local conditions}
\end{gathered}
\end{equation}

\end{document}